\documentclass[12pt]{article}
\usepackage{latexsym,amsmath,amscd,amssymb,framed,graphics,color}
\usepackage{enumerate}
\usepackage{graphicx}
\usepackage{url}
\usepackage[all]{xy}

\newcommand{\rem}[1]{}

\makeatletter

\@addtoreset{figure}{section}
\def\thefigure{\thesection.\@arabic\c@figure}
\def\fps@figure{h, t}
\@addtoreset{table}{bsection}
\def\thetable{\thesection.\@arabic\c@table}
\def\fps@table{h, t}
\@addtoreset{equation}{section}

\makeatother

\newcommand \al{\alpha}
\newcommand\be{\beta}
\newcommand\ga{\gamma}
\newcommand\de{\delta}

\newcommand\et{\eta}
\renewcommand\th{\theta}

\newcommand\la{\lambda}

\newcommand\si{\sigma}

\newcommand\ph{\varphi}

\newcommand\ps{\psi}
\newcommand\om{\omega}
\newcommand\Ga{\Gamma}

\newcommand\La{\Lambda}

\newcommand\Ph{\Phi}

\newcommand\Om{\Omega}

\newcommand\wrt{w.r.t.\ }
\newcommand\resp{resp.\ }
\newcommand\ie{i.e.\ }

\newcommand\oo{{\infty}}

\renewcommand\o{\circ}
\newcommand\vl{\on{vl}}
\newcommand\x{\times}
\newcommand\on{\operatorname}

\newcommand\ad{\on{ad}}

\newcommand\ev{\on{ev}}

\newcommand\pr{\on{pr}}
\newcommand\Fl{\on{Fl}}

\newcommand\F{\mathcal{F}}
\newcommand\Diff{\on{Diff}}

\newcommand\Hom{\on{Hom}}
\renewcommand\Im{\on{Im}}

\newcommand\g{\mathfrak g}

\newcommand\X{\mathfrak X}

\newcommand\CA{\mathcal{CA}}

\newcommand\SA{\mathcal{SA}}

\newcommand\ZZ{\mathbb Z}
\newcommand\RR{\mathbb R}
\newenvironment{proof}[1][Proof]{\noindent\textbf{#1.} }{\ \rule{0.5em}{0.5em}}

\newcommand{\fint}{-\!\!\!\!\!\!\int}

\def\XXint#1#2#3{{\setbox0=\hbox{$#1{#2#3}{\int}$ }
\vcenter{\hbox{$#2#3$ }}\kern-.5\wd0}}

\begin{document}

\newtheorem{theorem}{Theorem}[section]
\newtheorem{definition}[theorem]{Definition}
\newtheorem{lemma}[theorem]{Lemma}
\newtheorem{proposition}[theorem]{Proposition}
\newtheorem{corollary}[theorem]{Corollary}
\newtheorem{remark}[theorem]{Remark}
\newtheorem{example}[theorem]{Example}

\def\below#1#2{\mathrel{\mathop{#1}\limits_{#2}}}

\title{Current algebra functors and extensions}
\author{Anton Alekseev$^{1}$, Pavol Severa$^{2}$, and Cornelia Vizman$^{3}$ }

\addtocounter{footnote}{1}
\footnotetext{Section of Mathematics, University of Geneva, 2-4 Rue de Li\`evre, 
c.p. 64, 1211-Gen\`eve 4, Switzerland.
\texttt{Anton.Alekseev@unige.ch}
\addtocounter{footnote}{1} }
\footnotetext{Section of Mathematics, University of Geneva, 2-4 Rue de Li\`evre, 
c.p. 64, 1211-Gen\`eve 4, Switzerland. On leave from FMFI UK Bratislava, Slovakia.
\texttt{Pavol.Severa@unige.ch}
\addtocounter{footnote}{1} }
\footnotetext{Department of Mathematics,
West University of Timi\c soara, 4 Bd. V. P\^arvan, 300223-Timi\c soara, Romania.
\texttt{vizman@math.uvt.ro}
\addtocounter{footnote}{1} }

\date{ }
\maketitle

\makeatother
\maketitle




\begin{abstract}
We show how the fundamental cocycles on current Lie algebras
and the Lie algebra of symmetries for the sigma model are obtained via the current algebra functors
introduced in \cite{AlSe10}.
We present current group extensions integrating some of these current Lie algebra extensions.
\end{abstract}



\section{Introduction}

In \cite{AlSe10}, the first and second author introduced two {\em current algebra functors} $\CA$ and $\SA$
which assign  a Lie algebra (a current algebra) to a pair of a smooth manifold $S$ and a differential 
graded Lie algebra (dgla) $A$. Recall that for a Lie algebra $\g$ the cone $C\g$ is a dgla with a copy of $\g$ 
in degree zero, a copy of $\g$ in degree $(-1)$ and the differential $d: \g_{-1} \to \g_0$ mapping one copy to the 
other.  The name ``current algebra'' comes from the fact that for $A=C\g$ the current algebra functors 
yield the Lie algebra of maps from $S$ to $\g$,
$$
\CA(S, C\g) = \SA(S, C\g) = C^\infty(S, \g).
$$
Among other things, the current algebra functors are useful for constructing central and abelian
extensions of the Lie algebra $C^\infty(S, \g)$. In particular, in \cite{AlSe10} we showed how 
to derive the standard central extension of the loop algebra $C^\infty(S^1,\g)$, where $\g$ is a quadratic
Lie algebra. Surprisingly, it  comes from the central extension of
the dgla $C\g$ by a line in degree $(-2)$. This central extension is defined by formula
\begin{equation}  \label{xyc}
[x,y] = (x,y) c,
\end{equation}
where $x,y \in \g_{-1}$, $(\cdot, \cdot)$ is the invariant scalar product on $\g$, and $c$ is the 
generator of the central line. Note that equation \eqref{xyc} is the defining relation of the Clifford
algebra on $\g$ defined by the scalar product $(\cdot, \cdot)$. Another application of the current
algebra functors is an elegant construction of the Faddeev-Mickelsson-Shatashvili abelian
extension of the Lie algebra $C^\infty(S, \g)$ for $S$ a manifold of higher (odd) dimension.

The purpose of this note is to give two new applications of the current algebra functors.
The first application concerns the 2-cocycles on the Lie algebra $C^\infty(S, \g)$.
They were classified by Neeb in \cite{Ne07}. In Section 2, we show that all cocycles
of this classification can be obtained by applying the current algebra functors $\CA$ and $\SA$ 
to various extensions of the cone alga $C\g$. The key ingredient in the proof is an
observation that under some assumptions the current algebra functors map
exact sequences to exact sequences (even though this is not the case in general).

In Section 3, we present a new construction of current algebras associated to 2-dimensional $\sigma$-models
introduced in \cite{AlSt05} (for some applications in Physics, see \cite{nekrasov} and \cite{zabzine}). 
For these current algebras, the dgla $A$ is built from the Courant
bracket on the ``target'' manifold $M$. The main tool here is the study of  current algebras
associated to semi-direct products of dglas. Another important tool which allows to work with Poisson brackets
on $\sigma$-models is the hat calculus introduced in \cite{Vi09} and reviewed in Section 4.

\vskip 0.2cm

{\bf Acknowledgements.} We are grateful to K.-H. Neeb for his interest in our work.
This work was supported in part by the grants number 140985 and 126817 of the 
Swiss National Science Foundation, and by the grant PN-II-ID-PCE-2011-3-0921
of the Romanian National Authority for Scientific Research.


\section{Fundamental cocycles on current algebras}\label{fcca}

\subsection{Current algebra functors}

Two current algebra functors which associate Lie algebras to pairs $(S,A)$ consisting of 
a smooth manifold $S$ and a differential graded Lie algebra (dgla) $A$, were introduced in \cite{AlSe10}.
Their name comes from the fact that the current Lie algebra $C^\oo(S,\g)$ is obtained
when $A$ is the cone of a Lie algebra $\g$.
The cone $C\g$ of a Lie algebra $\g$ is the dgla spanned by elements $L(x)$ 
of degree $0$ and elements $I(x)$ of degree $-1$, with $x\in\g$. The Lie bracket on $C\g$ is
defined by $[L(x),L(y)]=L[x,y]$, $[L(x),I(y)]=I[x,y]$, $[I(x),I(y)]=0$,
while the differential on $C\g$ is given by $dI(x)=L(x)$.

The first functor is the current algebra functor $\CA$:
\begin{equation}\label{ca}
\CA(S,A)=(\Om(S)\otimes A)^{-1}/(\Om(S)\otimes A)_{exact}^{-1}
\end{equation}
with Lie bracket the derived bracket \cite{KS96}:
\begin{equation}\label{debra}
[\ph\otimes a,\ps\otimes b]_d=[\ph\otimes a,d(\ps\otimes b)]=\ph d\ps\otimes[a,b]+(-1)^{|\ps|}\ph\ps\otimes[a,db].
\end{equation}
The second current algebra functor $\SA$ is simply
\begin{equation}\label{sa}
\SA(S,A)=(\Om(S)\otimes A)^0_{closed}
\end{equation}
with the obvious Lie bracket $[\ph\otimes a,\ps\otimes b]=\ph\ps\otimes[a,b]$.
They fit into an exact sequence of Lie algebras:
\begin{equation}\label{patr}
0\to H^{-1}(\Om(S)\otimes A)\to\CA(S,A)\stackrel{d}{\to}\SA(S,A)\to H^0(\Om(S)\otimes A)\to 0.
\end{equation}

The two functors $\CA$ and $\SA$ coincide for acyclic dgla's. When $A$ is the cone $C\g$ of a Lie algebra,
both functors provide the current Lie algebra:
\begin{align*}\label{cas}
\CA(S,C\g)&=\mathrm{span}\{\ph\otimes I(x)\ :\ \ph\in C^\oo(S),x\in\g\}\cong C^\oo(S,\g)\\
\SA(S,C\g)&=\mathrm{span}\{\ph\otimes L(x)+d\ph\otimes I(x)\ :\ \ph\in C^\oo(S),x\in\g\}\cong C^\oo(S,\g).
\end{align*}

\begin{remark}\label{exact} {\rm \cite{AlSe10}
The current algebra functors $\CA$ and $\SA$ are contravariant \wrt to $S$ and covariant \wrt $A$.
Furthermore, for any short exact sequence of dgla's $0\to A\to B\to C\to 0$,
with $C$ acyclic as a complex, we get short exact sequences of Lie algebras:
\begin{gather*}
0\to\CA(S,A)\to\CA(S,B)\to\CA(S,C)\to 0\\
0\to\SA(S,A)\to\SA(S,B)\to\SA(S,C)\to 0.
\end{gather*}
}
\end{remark}


\subsection{Central dgla extensions of the cone}

In this section we show how the fundamental cocycles on current algebras
can be obtained by applying the current algebra functors to central dgla extensions of the cone of a Lie algebra.

\begin{theorem}\label{fundam}{\rm  \cite{Ne07}}
Each continuous 2-cocycle on the current Lie algebra $C^\oo(S,\g)$ is a sum of cocycles that factor through 
a cocycle of the following type:
\begin{align*}
1.\ &\si_{p}(u,v)=p(u,dv)+dC^\oo(S)\in\Om^1(S)/dC^\oo(S),\quad p\in (S^2\g^*)^\g;\\
2.\ &\si_N(u,v)=p(u,dv)-p(du,v)+d(\om(u,v))\in\Om^1(S),\\&  p\in(S^2\g^*)^\g \text{ and } 
\om\in\La^2\g^* \text{ such that }p([\ ,\ ],\ )=-d_\g\om;\\
3.\ &\si_\ga(u,v)=\ga(u,v)\in C^\oo(S),\quad\ga\in\La^2\g^*,\ d_\g\ga=0;\\
4.\ &\text{pull-back of a cocycle on an abelian quotient of }C^\oo(S,\g).
\end{align*}
\end{theorem}

\paragraph{One-dimensional central extensions of the cone.}
The isomorphism classes of central one-dimensional graded Lie algebra extensions of $C\g$  by $\RR[k]$ are classified by
\begin{align*}
k=0:\ (L(x)+I(y),L(x')+I(y'))&\mapsto\rho(x,x'),\ \quad\quad\quad\quad [\rho]\in H^2(\g)\nonumber\\ 
k=1:\ (L(x)+I(y),L(x')+I(y'))&\mapsto \la(x)y'-\la(x')y,\quad [\la]\in H^1_{\ad^*}(\g,\g^*)\\
k=2:\ (L(x)+I(y),L(x')+I(y'))&\mapsto p(y,y'),\quad\quad\quad\quad\quad p\in(S^2\g^*)^\g\nonumber.
\end{align*}
We get a nontrivial dgla extension of $C\g$ by $\RR[k]$ if either $k=2$, or if $k=1$ and the cocycle $\la:\g\to\g^*$ is such that the map 
$(x,y)\mapsto\la(x)y=\ga(x,y)$ is skew-symmetric on $\g$. In both cases the extended dgla is $C\g\oplus\RR[k]$ as a complex.

Applying the current algebra functor $\CA$ to the corresponding 
dgla cone extensions denoted by $C_p\g$ and $C_\ga\g$, and using remark \ref{exact},
the central extensions of the current algebra $C^\oo(S,\g)$ by $C^\oo(S)$, \resp by $\Om^1(S)/dC^\oo(S)$ are obtained: 
\begin{align*}
\CA(S,C_\om\g)&=C^\oo(S,\g)\x_{\si_\ga} C^\oo(S),\quad\quad\quad\quad\ \si_\ga(u,v)=\ga(u,v)\\
\CA(S,C_p\g)&=C^\oo(S,\g)\x_{\si_{p}}\Om^1(S)/dC^\oo(S),\quad \si_{p}(u,v)=p(u,dv)+dC^\oo(S).
\end{align*}
These are the cocycles  1. and 3. from Theorem \ref{fundam}.

We check it for the cocycle $\si_\ga$, since the cocycle $\si_p$ has been treated in \cite{AlSe10}.
The derived bracket \eqref{debra} in $\CA(S,C_\ga\g)$ is
$$
[\ph\otimes I(x),\ps\otimes I(y)]_d=[\ph\otimes I(x),\ps\otimes L(y)+d\ps\otimes I(y)]=\ph\ps\otimes I[x,y]
+\ph\ps\otimes\la(x)y.
$$
Then the characteristic cocycle of the central extension is
\begin{align*}
\si_\ga(\ph\otimes x,\ps\otimes y)
\stackrel{\eqref{cas}}{=}\si_\ga\left(\ph\otimes I(x),\ps\otimes I(y)\right)
=\ph\ps\otimes\la(x)y=\ph\ps\otimes\ga(x,y),
\end{align*}
which can be rewritten as $\si_\ga(u,v)=\ga(u,v)$ for all $u,v\in C^\oo(S,\g)$.

\medskip

To get the Neeb cocycle $\si_N$, the second cocycle from Theorem \ref{fundam}, one needs a two dimensional dgla central extension of $C\g$.
First we notice that a $\g^*$-valued 1-cocycle $\al$ on $\g$ (for the coadjoint action)  is the same
as an exact invariant inner product $p$ on $\g$.
Indeed, any such cocycle $\al$ decomposes into its symmetric part
$p\in S^2\g^*$ and skew-symmetric part $\om\in\La^2\g^*$ by
$\al(x)y=p(x,y)+\om(x,y)$, $x,y\in\g$.
One gets that 
$$
(d_\g\al)(x,y)z=p(x,[y,z])+(d_\g\om)(x,y,z),
$$
so that the vanishing of the Chevalley-Eilenberg differential of $\al$ is equivalent to
the exactness of $p$, namely $p([\ ,\ ],\ )=-d_\g\om$. 
From this also the $\g$-invariance of $p$ follows.

The 2-cocycle on $C\g$ 
\begin{equation}\label{alphacocycle}
(L(x)+I(y),L(x')+I(y'))\mapsto (\al(x)y'-\al(x')y,\al(y)y'+\al(y')y)
\end{equation}
with values in $\RR[1]\oplus\RR[2]$
defines a 2-dimensional central graded Lie algebra extension of $C\g$, denoted by $C_\al\g$.
The differential that takes the central element $c_2$ in degree $-2$ to the central element $c_1$ in degree $-1$
(and of course $I(x)$ to $L(x)$) makes it  a dgla.

\begin{proposition}
The Lie algebra $\SA(S,C_\al\g)$ is isomorphic to the central Lie algebra extension of the current algebra $C^\oo(S,\g)$ by $\Om^1(S)$ defined with the Neeb cocycle $\si_N$.
\end{proposition}

\begin{proof}
The Lie bracket on
\begin{multline*}
\SA(S,C_\al\g)=\{df\otimes I(x)+f\otimes L(x)+\be\otimes c_1+d\be\otimes c_2:\\
 f\in C^\oo(S),\ \be\in\Om^1(S),\ x\in\g\}
\end{multline*}
is seen to be
\begin{multline*}
[df\otimes I(x)+f\otimes L(x),df'\otimes I(x')+f'\otimes L(x')]\\
\stackrel{\eqref{alphacocycle}}{=}d(ff')\otimes I({[x,x']})+ff'\otimes L([x,x'])
+df\wedge df'\otimes(\al(x)x'\\
+\al(x')x)c_2+fdf'\otimes(\al(x)x')c_1-f'df\otimes(\al(x')x)c_1.
\end{multline*}
Using also the decomposition of $\al$ into its symmetric and skew-symmetric parts $p$ and $\om$,
the characteristic cocycle is 
\begin{align*}
\si_\al(df\otimes I(x)&+f\otimes L(x),df'\otimes I(x')+f'\otimes L(x'))=
df\wedge df'\otimes 2p(x,x')c_2\\
&+\left(fdf'\otimes p(x,x')-f'df\otimes p(x,x')+d(ff')\otimes\om(x,x')\right)c_1.
\end{align*}
Written for $u,v\in C^\oo(S,\g)$, the $c_1$ component of $\si$ is Neeb cocycle 
$\si_N(u,v)=p(u,dv)-p(du,v)+d(\om(u,v))\in\Om^1(S)$, while the $c_2$ component is 
its differential $d\si_N(u,v)=2p(du,dv)\in\Om^2(S)$.
\end{proof}

\begin{proposition}
If the 1-cocycle $\al:\g\to\g^*$ involved in Neeb cocycle can be integrated to a group 1-cocycle $a:G\to\g^*$,
then there is a group 2-cocycle $c$ on the current group $C^\oo(S,G)$, integrating the Neeb cocycle:
$$
c(g_1,g_2)=a(g_1)(dg_2g_2^{-1})\in\Om^1(S).
$$
\end{proposition}
\begin{proof}
The Lie algebra cocycle associated to the group cocycle $c$ is 
$\al(u)dv-\al(v)du=p(u,dv)-p(v,du)+d\om(u,v)=\si_N(u,v)$ for all $u,v\in C^\oo(S,\g)$.
\end{proof}

It is shown in \cite{Ne07} that for any simply connected Lie group $G$ integrating $\g$,
the current algebra extension defined with the Neeb cocycle
can be integrated to a current group extension. This is a special case of the proposition above,
since any Lie algebra 1-cocycle $\al$ on $\g$
can be integrated to a group 1-cocycle $a$ on the simply connected Lie group $G$.


\section{Sigma model symmetries via current algebra functors}\label{aeca}

\subsection{Abelian extensions of current algebras}

A differential graded space $V$ with an action of the dgla $C\g$, 
compatible with the differential (and with the grading) on $V$, 
is called a $\g$-differential space.
This means that $L(x)v=I(x)dv+dI(x)v$ for all $x\in\g$ and $v\in V$.
The main example of a differential graded space is 
the space of differential forms $V=\Om(M)$ on a $\g$-manifold $M$.

The semidirect product $C\g\ltimes V$ is again a dgla. By remark \ref{exact}, since $C\g$ is acyclic, the current algebra functors 
$\CA$ and $\SA$ applied to $C\g\ltimes V$  provide semidirect products of the current Lie algebra $C^\oo(S,\g)$.
For instance $\CA(S,C\g\ltimes V)=C^\oo(S,\g)\ltimes \CA(S,V)$,
for the natural (derived) action of the current algebra $\CA(S,C\g)=C^\oo(S,\g)$ on $\CA(S,V)=(\Om(S)\otimes V)^{-1}/(\Om(S)\otimes V)^{-1}_{exact}$:
\begin{align}\label{deac}
(\ph\otimes I(x))\cdot_d(\et\otimes v)&=d(\ph\otimes I(x))\cdot(\et\otimes v)\\
&=(\ph\otimes L(x)+d\ph\otimes I(x))\cdot(\et\otimes v)\nonumber\\
&=\ph\et\otimes L(x)v+(-1)^{|\et|}(d\ph\wedge\et)\otimes I(x)v,
\end{align}
for $\ph\in C^\oo(S)$ and $\et\in\Om(S)$. 

More generally, the current algebra functors provide abelian current algebra extensions when
applied to abelian dgla extensions of the cone $C\g$ by a $\g$-differential space $V$.
These are described by pairs
$(\om,\de)$, where $\om\in\La^2(C\g)^*\otimes V$ deformes the Lie bracket
and $\de\in (C\g)^*\otimes V$ deformes the differential,
satisfying the closedness condition $d_{tot}(\om+\de)=0$, where $d_{tot}$ is the sum of the 
Chevalley-Eilenberg differential $d_{C\g}$ for the Lie algebra $C\g$ and the differential $\bar d$
is induced by the differential graded spaces $C\g$ and $V$. 
This closedness condition is equivalent to 
\[
d_{C\g}\om=0,\quad d_{C\g}\de+\bar d\om=0,\quad \bar d\de=0.
\]
In particular $\om$ is a  2-cocycle on the cone.

One obtains the dgla $(C\g\ltimes_\om V)_{(\de)}$ with differential $\tilde d(a,v)=(da+\de a,dv)$
and Lie bracket 
$[(a,v),(b,w)]=([a,b],a\cdot w-b\cdot v+\om(a,b))$ for $a,b\in C\g$, so
\begin{align}\label{dtil}
\tilde dI(x)=L(x)+\de I(x),\quad \tilde dL(x)=\de L(x).
\end{align}

\begin{theorem}\label{deom}
The current algebra functor $\CA$ applied to the dgla $(C\g\ltimes_\om V)_{(\de)}$ provides 
an abelian Lie algebra extension of the current algebra $C^\oo(S,\g)$ 
with characteristic $\CA(S,V)$-valued cocycle 
\begin{align}\label{sod}
\si_{(\om,\de)}(u,v)= &\frac12\big(I(u)\cdot\de I(v)-I(v)\cdot\de I(u)\\
+&\om(I(du),I(v))-\om(I(dv),I(u))
+\om(I(u),L(v))-\om(I(v),L(u))\big).\nonumber
\end{align}
\end{theorem}

\begin{proof}
We compute the $\tilde d$-derived bracket, \ie the bracket on $\CA(S,(C\g\ltimes_\om V)_{(\de)})$:
\begin{align*}
[\ph\otimes I(x),\ps\otimes I(y)]_{\tilde d}&=[\ph\otimes I(x),
\ps\otimes(L(y)+\de I(y))+d\ps\otimes I(y)]\\
&=\ph\ps\otimes I([x,y])+\ph\ps\otimes I(x)\cdot\de I(y)\\
&+\ph\ps\otimes\om(I(x),L(y))+\ph d\ps\otimes\om(I(x),I(y)). 
\end{align*}
Using the identity which follows from $d_{C\g}\de+d\om=0$:
\begin{gather*}
d(\ph\ps\otimes\om(I(x),I(y)))=d(\ph\ps)\otimes\om(I(x),I(y))\\
+\ph\ps\otimes( I(x)\cdot\de I(y)-I(y)\cdot\de I(x)
+\om(I(x),L(y))+\om(I(y),L(x))),
\end{gather*}
after we mod out the exact term $\frac12d(\ph\ps\otimes\om(I(x),I(y)))$, one gets the cocycle
\begin{gather*}
\si_{(\om,\de)}(\ph\otimes I(x),\ps\otimes I(y))
=\frac12(\ph d\ps-\ps d\ph)\otimes\om(I(x),I(y))\\
+\frac12\ph\ps\otimes( I(x)\cdot\de I(y)-I(y)\cdot\de I(x)
+\om(I(x),L(y))-\om(I(y),L(x)))
\end{gather*}
that can be rewritten as \eqref{sod} for $u,v\in C^\oo(S,\g)$.
\end{proof}

A similar result holds for the $\SA$ functor (see also the exact sequence of Lie algebras \eqref{patr}):

\begin{proposition}\label{cor}
The current algebra functor $\SA$ applied to the dgla $(C\g\ltimes_\om v)_{(\de)}$ provides 
an abelian Lie algebra extension of the current algebra $C^\oo(S,\g)$ 
with characteristic $\SA(S,V)$-valued cocycle given by the differential
of the cocycle from theorem \ref{deom}.
\end{proposition}

First we present an example considered in \cite{AlSe10}, where $\om=0$.
Let $e\in V^1$ such that $p=de$ is basic with respect to the $C\g$ action, 
\ie $I(x)p=L(x)p=0$ for all $x\in\g$.
We define $\de(I(x))=-I(x)e$ and $\de(L(x))=L(x)e$
and we deform the differential on $C\g\ltimes V$ as in \eqref{dtil}  to
\begin{equation}\label{tilded}
\tilde dI(x)=L(x)-I(x)e,\quad\tilde dL(x)=L(x)e.
\end{equation}
This new dgla is denoted by $C_e\g=(C\g\ltimes V)_{(e)}$

\begin{corollary}\label{sige}
The current algebra functor $\CA$ applied to the dgla $C_e\g$ provides 
an abelian Lie algebra extension of the current algebra $C^\oo(S,\g)$ 
with characteristic 2-cocycle 
\begin{equation}\label{unu}
(u,v)\mapsto -I(u)I(v)e\in \CA(S,V),
\end{equation}
while the characteristic cocycle for the  current algebra functor $\SA$ applied to $C_e\g$ is 
\begin{equation}\label{doi}
(u,v)\mapsto -dI(u)I(v)e\in \SA(S,V).
\end{equation}
\end{corollary}

\begin{proof}
The first formula follows from theorem \ref{deom}: 
\begin{align*}
\si_{(\om,\de)}(u,v)&=\tfrac12\left(I(u)\cdot\de I(v)-I(v)\cdot\de I(u)\right)
=\tfrac12\left(I(v)I(u)e-I(u)I(v)e\right)\\
&=-I(u)I(v)e,
\end{align*}
while for the second formula we just apply Proposition \ref{cor}.
\end{proof}

\begin{example}
\rm{
The Fadeev-Mickelsson-Shatashvili (FMS) Lie algebra is a 1-dimensional extension of a current algebra extension 
by the module of functionals on the space of connections, 
and can be obtained by the current algebra functor $\CA$ as explained in \cite{AlSe10}.
We present here a truncated FMS extension, as application of theorem \ref{deom}.

The cone $C\g^*$ of the Lie algebra dual $\g^*$,
with elements denoted by $\ell(\xi)$ and $\iota(\xi)$  for $\xi\in\g^*$, is 
a $\g$-differential space for the $C\g$-action: 
\begin{align*}
L(x)\cdot\ell(\xi)&=\ell(\ad_x^*\xi)\quad\quad I(x)\cdot\ell(\xi)=\iota(\ad_x^*\xi)\\
L(x)\cdot\iota(\xi)&=\iota(\ad_x^*\xi)\quad\quad I(x)\cdot\iota(\xi)=0.
\end{align*}
Each element $p\in(S^3\g^*)^\g$ provides a Lie algebra cocycle $\om$ on $C\g$
\[
\om(I(x),I(y))=(I(x),I(y))\mapsto \ell(p(x,y,\cdot))\in C\g^*[2],
\]
which defines  a dgla extension $B=C\g\ltimes_\om C\g^*[2]$ for
the natural differential $dI(x)=L(x)$ and $d\iota(\xi)=\ell(\xi)$.

We apply theorem \ref{deom} to the pair $\om$ and $\de=0$ and we get that the abelian Lie algebra extension 
$\CA(S,B)$  of the current algebra $C^\oo(S,\g)$ by the module $\CA(S,C\g^*[2])$,
described by the characteristic cocycle 
\begin{multline*}
\si(u,v)=\tfrac12\left(\om(I(du),I(v))-\om(I(dv),I(u))\right)\\
=\tfrac12 \ell\left(p(du,v,\cdot)-p(dv,u,\cdot)\right)=
-i(p(du,dv,\cdot))+\tfrac12 d(i(p(du,v,\cdot)-p(dv,u,\cdot))).
\end{multline*}
When $S$ is 3-dimensional, $\CA(S,C\g^*[2])=\Om^2(S,\g^*)$
is the space of  linear functionals on the space  $\Om^1(S,\g)$ of $\g$-valued 1-forms, and we get the  FMS cocycle:
\[
\si(u,v)[A]= \int_S p(du,dv,A),\quad A\in\Om^1(S,\g).
\]
The dgla $B_{FMS}$ is the extension of the dgla $B=C\g\ltimes_\om C\g^*[2]$ by $\RR[4]$, given by the cocycle
$(I(x),\iota(\xi))\mapsto \xi(x)$.
Then $\CA(S,B_{FMS})$ is the truncated FMS current algebra, namely the central extension 
of  $\CA(S,B)=C^\oo(S,\g)\ltimes_\si\Om^2(S,\g^*)$ by $\RR$, 
with cocycle defined by
$$(f,\ga)\in C^\oo(S,\g)\x\Om^2(S,\g^*)\mapsto (d\ga,f)\in\Om^3(S)/\Om^3(S)_{exact}=\RR.$$ 
}
The abelian Lie subalgebra $\Om^2(S,\g^*)\oplus\RR$ of $\CA(S,B)$ is the space of (inhomogeneous) linear functions on the affine space $\Om^1(S,\g)$ of $\g$-valued connections. 

\end{example}


\subsection{Sigma model symmetries}

In this section we show that the Lie algebra of symmetries for the sigma model \cite{AlSt05} 
can be obtained with the current algebra functor $\CA$ applied to a semidirect product dgla with deformed differential, 
as in corollary \ref{sige}.

\paragraph{Cotangent bundle of manifolds of smooth maps.}
Let $\F=C^\oo(S,M)$ be the Fr\'echet manifold of smooth maps, where $S$ is a compact $k$-dimensional manifold.
By the hat calculus presented in the appendix, Section \ref{hatcal},
a closed $(k+2)$-form $H$ on $M$ defines a closed 2-form $\hat H=\widehat{H\cdot 1}$
on the manifold $C^\oo(S,M)$. This means that $\hat H(u_f,v_f)
=\int_Si_{v_f}i_{u_f}H$ for all $u_f,v_f$ vector fields on $M$ along $f\in\F$.
The 2-form $\hat H$ can be used to twist the canonical symplectic form $\om$ on the cotangent bundle 
$p:T^*\F\to\F$ in the usual way $\om_{\hat H}=\om+p^*\hat H$.
We denote by $\{\ ,\ \}_{\hat H}$ the corresponding Poisson bracket.
Section \ref{cotg} in the appendix collects some facts about the twisted cotangent bundle,
that will be used below for $T^*\F$.

The function on the cotangent bundle $T^*\F$
$$
F_{\et\otimes\al}=p^*\widehat{\al\cdot\et}
$$ 
is associated to differential forms $\et$ on $S$ and $\al$ on $M$ with $|\al|+|\et|=k$.
Here $\widehat{\al\cdot\et}(f)=\int_S f^*\al\wedge\et$ is a function on $\F$ obtained by the hat calculus. 
The identity $\widehat{(d\be)\cdot\et}+(-1)^{|\be|}\widehat{\be\cdot d\et}=0$ for $|\be|+|\et|+1=k$
implies the relation 
\begin{equation}\label{re}
F_{\et\otimes d\be}+(-1)^{|\be|}F_{d\et\otimes\be}=0.
\end{equation}

Another function on the cotangent bundle $T^*\F$
$$
J_{\ph\otimes v}=\ev_{\ph\bar v}
$$ 
is associated to a function $\ph\in C^\oo(S)$ and a vector field $v\in\X(M)$,
namely it is the function on $T^*\F$ defined by evaluation at the vector field 
$\ph\bar v:f\mapsto \ph(v\o f)$ on $\F$, as in remark \ref{gene}.

There is an injective Lie algebra homomorphism (Proposition \ref{cinci})
\[
\Ph:Y+f\in \X(\F)\ltimes_{\hat H} C^\oo(\F)\mapsto \ev_Y+p^*f\in (C^\oo(T^*\F),\{\ ,\ \}_{\hat H}).
\]
Notice that the functions $F_{\et\otimes\al}=\Ph(\widehat{\al\cdot\et})$ for $\widehat{\al\cdot\et}\in C^\oo(\F)$
and $J_{\ph\otimes v}=\Ph(\ph\bar v)$ for $\ph\bar v\in\X(\F)$ all belong to 
the Lie subalgebra $\Im\Ph$ which is described in the next proposition.

\begin{proposition}\label{pbra}
The  Poisson brackets of functions of type $F_{\et\otimes\al}$ and $J_{\ph\otimes v}$ on the cotangent bundle $T^*C^\oo(S,M)$ with twisted symplectic form $\om_{\hat H}=\om+p^*\hat H$ are:
\begin{enumerate}
\item $\{J_{\ph\otimes v},J_{\ps\otimes w}\}_{\hat H} =J_{\ph\ps\otimes[v,w]}-F_{\ph\ps\otimes i_vi_wH}$;
\item $\{J_{\ph\otimes v},F_{\et\otimes\al}\}_{\hat H} =F_{\ph\et\otimes L_v\al}-(-1)^{|\al|}F_{d\ph\wedge\et\otimes i_v\al}$;
\item $\{F_{\et\otimes\al},F_{\et'\otimes\al'}\}_{\hat H} =0$.
\end{enumerate}
\end{proposition}

\begin{proof}
For the proof we use the fact that $\Ph$ is a Lie algebra homomorphism,
as well as the hat calculus from section \ref{hatcal}.
\begin{align*}
\{J_{\ph\otimes v},J_{\ps\otimes w}\}_{\hat H}
&=\{\Ph(\ph \bar v),\Ph(\ps\bar w)\}_{\hat H}
=\Ph([\ph\bar v,\ps\bar w])+\Ph(\hat H(\ph\bar v,\ps\bar w))\\
&=\Ph(\ph\ps\overline{[v,w]})+\Ph(\widehat{i_vi_wH\cdot\ph\ps})
=J_{\ph\ps\otimes[v,w]}+F_{\ph\ps\otimes i_vi_wH}.
\end{align*}
The next computation uses the identity \eqref{tria} at step four
\begin{align*}
\{J_{\ph\otimes v},F_{\et\otimes\al}\}
&=\{\Ph(\ph\bar v),\Ph(\widehat{\al\cdot\et})\}_{\hat H}
=\Ph([\ph\bar v,\widehat{\al\cdot\et}])
=\Ph(L_{\ph\bar v}\widehat{\al\cdot\et})\\
&=\Ph(\widehat{L_v\al\cdot\ph\et})-(-1)^{|\al|}\Ph(\widehat{i_v\al\cdot(d\ph\wedge\et)})\\
&=F_{\ph\et\otimes L_v\al}-(-1)^{|\al|}F_{d\ph\wedge\et\otimes i_v\al}.
\end{align*}
The third identity is an immediate consequence of the fact that $\widehat{\al\cdot\et},\widehat{\al'\cdot\et'}\in C^\oo(\F)$
lie in the abelian part of the abelian Lie algebra extension by $\hat H$.
\end{proof}
This is the higher dimensional version of the Lie algebra of symmetries for the sigma model \cite{AlSt05},
namely when the circle is replaced by a $k$-dimensional compact manifold $S$.

\begin{remark} {\rm
In the special case when  $S$ is a circle and $H$ a closed 3-form, 
one gets the following generating functions on the cotangent bundle $T^*LM$
of the loop space $LM$: $J_{\ph\otimes v}, F_{\ph\otimes\al}, F_{\et\otimes f}$,
where $\ph\in C^\oo(S)$,  $v\in\X(M)$, $\et\in\Om^1(S)$, $f\in C^\oo(M)$, $\al\in\Om^1(M)$. 
In \cite{AlSt05} the first two functions were put together  into a function $J_{\ph\otimes(v,\al)}$,
where $(v,\al)\in\Ga(TM\oplus T^*M)$ is a section of the Courant algebroid $TM\oplus T^*M$.
Their bracket involve the twisted Courant bracket
\[
[(v_1,\al_1),(v_2,\al_2)]=([v_1,v_2],L_{v_1}\al_2-i_{v_2}d\al_1+i_{v_1}i_{v_2}H),
\] 
but provide an "anomalous" term involving the inner product
\[
((v_1,\al_1),(v_2,\al_2))=i_{v_1}\al_2-i_{v_2}\al_1\in C^\oo(M),
\] 
because
\[
\{J_{\ph_1\otimes(v_1,\al_1)},J_{\ph_2\otimes(v_2,\al_2)}\}=J_{\ph_1\ph_2\otimes[(v_1,\al_1),(v_2,\al_2)]}
+F_{\ph_1d\ph_1\otimes((v_1,\al_1),(v_2,\al_2))}.
\]
It works similarly for $k>1$: on $\Ga(TM\oplus\La^kT^*M)$ there is a bracket
and an inner product with values in $\Om^{k-1}(M)$, both defined by the same formulas as above.
The anomalous term $F_{\ph_1d\ph_1\otimes((v_1,\al_1),(v_2,\al_2))}$ comes 
now from $\Om^1(M)\otimes\Om^{k-1}(M)$.
}
\end{remark}


\paragraph{An abelian current algebra extension.}
Let $S$ be a smooth $k$-dimensional manifold and $M$ a $\g$-manifold. 
Then the space $V=\Om(M)[k+1]$ of differential forms on $M$ with a $(k+1)$ shift in degree
is a $\g$-differential space. We have $I(x)\al=i_{x_M}\al$ and $L(x)\al=L_{x_M}\al$,
where $x_M\in\X(M)$ denotes the infinitesimal generator  for $x\in\g$.
 
The closed $(k+2)$-form $H$ on $M$ is an element in $V^1$, so it can be used for deforming 
the derivation on $C\g\ltimes V$ as in \eqref{tilded}:
\[
\tilde dI(x)=L(x)-i_{x_M}H,\quad \tilde dL(x)=L_{x_M}H.
\]
One obtains the dgla 
$(C\g\ltimes \Om(M)[k+1])_{(H)}$.

\begin{proposition}\label{prin}
The Lie algebra $\CA(S,(C\g\ltimes \Om(M)[k+1])_{(H)})$  is an abelian extension
of the current algebra $C^\oo(S,\g)$ by the module
\begin{equation*}
\CA(S,\Om(M)[k+1])=\left(\sum_{l=0}^k\Om^l(S)\otimes\Om^{k-l}(M)\right)\Big/
d\left(\sum_{l=0}^{k-1}\Om^l(S)\otimes\Om^{k-l-1}(M)\right),
\end{equation*}
with characteristic 2-cocycle 
$$
\si_H(u,v)=i_{v_M}i_{u_M}H,\quad u,v\in C^\oo(S,\g).
$$
The Lie algebra bracket written on its generators $\ph\otimes I(x)\in C^\oo(S,\g)$  
and $\et\otimes \al\in\CA(S,\Om(M)[k+1])$, where  $\ph\in C^\oo(S)$,  $x\in\g$, $\et\in\Om^l(S)$, $\al\in\Om^{k-l}(M)$, $l=0,\dots,k$,
are
\begin{align*}
[\ph\otimes I(x),\ps\otimes I(y)]&=\ph\ps\otimes I([x,y])+\ph\ps\otimes i_{x_M}i_{y_M} H\\
[\ph\otimes I(x),\et\otimes \al]&=\ph\et\otimes L_{x_M}\al+(-1)^{|\et|}d\ph\wedge\et\otimes i_{x_M}\al\\
[\et\otimes \al,\et'\otimes \al']&=0
\end{align*}
with relations 
$$d(\et\otimes\be)=d\et\otimes\be+\et\otimes d\be=0.$$ 
\end{proposition}

\begin{proof}
The expression of the cocycle $\si_H$ on the current algebra follows from Corollary \ref{sige}. 
The first identity involves the cocycle $\si_H$, the second identity
involves the derived action \eqref{deac} of $C^\oo(S,\g)$ on $\CA(S,\Om(M)[k+1])$, 
while the third identity reflects the fact that $\CA(S,\Om(M)[k+1])$ is abelian.
\end{proof}

\begin{remark}{\rm
When $\g=\X(M)$, the Lie algebra $\CA(S,(C\g\ltimes \Om(M)[k+1])_{(H)})$ from Proposition \ref{prin} 
is isomorphic to the Lie algebra of sigma model symmetries from Proposition \ref{pbra},
through the identifications:
$J_{\ph\otimes v}$ with $\ph\otimes I(v)$, 
and $F_{\et\otimes\al}$ with $(-1)^{|\et||\al|}\et\otimes\al$.
}
\end{remark}

\begin{theorem}
Let $H\in\Om^{k+2}(M)$ be a closed differential form. 
There is an isomorphism between the Lie algebra obtained with the current functor $\CA$  applied 
to the dgla $(C\X(M)\ltimes\Om(M)[k+1])_{(H)}$
and the Lie algebra of sigma model symmetries, \ie the Lie algebra of functions on the symplectic manifold 
$(T^*C^\oo(S,M),\ \om_{\hat H}=\om+p^*\hat H)$ generated by 
$J_{\ph\otimes v}=\ph\bar v$ and 
$F_{\et\otimes\al}=p^*\widehat{\al\cdot\et}$, 
where $\ph\in C^\oo(S)$,  $v\in\X(M)$, $\et\in\Om^l(S)$, $\al\in\Om^{k-l}(M)$, $l=0,\dots,k$.

In particular the Lie algebra of sigma model symmetries  is an abelian extension
of the current algebra $C^\oo(S,\X(M))$ by the module 
$$\CA(S,\Om(M)[k+1])=\Om^k(S\x M)/\Om^k(S\x M)_{exact},$$
defined by the 2-cocycle 
\begin{equation}\label{sig}
\si_H(u,v)=i_{v}i_{u}H,\quad u,v\in C^\oo(S,\X(M)). 
\end{equation}
\end{theorem}


\paragraph{Current group extension.}
As it is observed in \cite{Ne07}, an obvious source of 2-cocycles on Lie algebras of vector fields
are the closed differential forms. Let $G$ be a Lie group acting on a smooth manifold $N$
and let $x_N\in \X(N)$ be the infinitesimal generator for $x\in\g$. 
Then for any closed differential form $\om\in\Om^{k+2}(N)$,
\[
\om^{[2]}:\g\x\g\to\Om^k(N)/\Om^k(N)_{exact},\quad \om^{[2]}(x,y)=i_{y_N}i_{x_N} H
\]
is a Lie algebra 2-cocycle.

\begin{theorem}\label{kh}{\rm \cite{Ne07}}
Whenever the cohomology class of the closed form $\om$ is integral,
the Lie algebra extension defined by the cocycle $\om^{[2]}$ integrates to an abelian Lie group extension
of the universal covering of $G_0$, the identity component of $G$,
by the quotient group $(\Om^k(N)/\Om^k(N)_{exact})/H^k(N,\ZZ)$, where
\[
H^k(N,\ZZ)=\Hom(H_k(N),\ZZ)=\left\{[\al]\in H^k_{dR}(N):\int_{Z_k(N)}\al\in\ZZ\right\}.
\]
\end{theorem}

The proof uses a general theorem on the two obstructions for the integrability of abelian Lie algebra extensions
to extensions of infinite dimensional Lie groups \cite{Ne04}.


Theorem \ref{kh} can be applied to our cocycle $\si_H$ from \eqref{sig}.
The Lie group $G=C^\oo(S,\Diff(M))$ acts in a natural way on $N=S\x M$, namely $g\cdot(s,m)=(s,g(s)(m))$.
If $\om$ denotes pull-back of the closed form $H\in\Om^{k+2}(M)$ to $S\x M$, 
then the cocycle $\si_H$ is just the restriction of $\om^{[2]}$ 
to the current algebra $\g=C^\oo(S,\X(M))$, a Lie subalgebra of $\X(S\x M)$.

\begin{corollary}
Assuming  that the cohomology class of the closed $(k+2)$-form $H$ on $M$ is integral,
the Lie algebra of sigma model symmetries
(an abelian Lie algebra extension
of $C^\oo(S,\X(M))$ by the module $\Om^k(S\x M)/\Om^k(S\x M)_{exact}$
with cocycle $\si_H$)
integrates to an abelian Lie group extension of the universal covering of the identity component 
$C^\oo(S,\Diff(M))_0$ by the group $(\Om^k(S\x M)/\Om^k(S\x M)_{exact})/H^k(S\x M,\ZZ)$.
\end{corollary}


\section{Appendix}


\subsection{Hat calculus}\label{hatcal}

The space $C^\oo(S,M)$ of smooth functions 
on the compact $k$--dimensional manifold $S$ with values in the manifold $M$
is a Fr\'echet manifold in a natural way. The tangent space at $f\in C^\oo(S,M)$
is identified with the space of smooth sections of the pullback bundle $f^*TM$
(vector fields on $M$ along $f$), \ie $v_f:s\in S\mapsto v_f(s)\in T_{f(s)}M$.

Differential forms on $C^\oo(S,M)$ can be obtained in a natural way 
from pairs of differential forms on $M$ and $S$ by the hat pairing \cite{Vi09}.
Let $\ev:S\x C^\oo(S,M)\to M$ be the evaluation map 
$\ev(s,f)=f(s)$, and $\pr:S\x C^\oo(S,M)\to S$ the projection 
on the first factor.
A pair of differential forms $\al\in\Om^p(M)$ and $\et\in\Om^q(S)$,
with $p+q\geq k$, determines a differential form 
$\widehat{\al\cdot\et}$ on $C^\oo(S,M)$:
\begin{equation}\label{use}
{\widehat{\al\cdot\et}=\fint_S\ev^*\al\wedge\pr^*\et}.
\end{equation}
It is the fiber integral over $S$ of the $(p+q)$-form 
$(\ev^*\al\wedge\pr^*\et)$ on $S\x C^\oo(S,M)$.
In this way one obtains a bilinear map called the hat pairing: 
\begin{equation*}\label{pair}
\Om^p(M)\x\Om^q(S)\to\Om^{p+q-k}(C^\oo(S,M)),\quad (\al,\eta)\mapsto \widehat{\al\cdot\et}.
\end{equation*}
A special case is the hat pairing of differential forms on $M$ with the constant function $1\in\Om^0(S)$ 
(the transgression map)
\begin{equation}\label{tran}
\al\in\Om^p(M)\mapsto\hat\al=\fint_S\ev^*\al\in\Om^{p-k}(C^\oo(S,M)).
\end{equation}
If $p+q=k$, then the result of the hat pairing is a function on $C^\oo(S,M)$:
\begin{equation}\label{func}
\widehat{\al\cdot\et}(f)=\int_S f^*\al\wedge\et, \text{ for all }f\in C^\oo(S,M).
\end{equation}

The diffeomorphism group $\Diff(M)$ acts naturally from the left on $C^\oo(S,M)$. 
The infinitesimal generator of $v\in\X(M)$ (the Lie algebra of the diffeomorphism group
is the Lie algebra of vector fields with the opposite bracket) 
is the vector field $\bar v$ on $C^\oo(S,M)$ given by
\[
\bar v(f)=v\o f,\quad\text{for all}\quad  f\in C^\oo(S,M).
\] 
The bracket of two infinitesimal generators is 
$\left[\bar v,\bar w\right]=\overline{[v,w]}$. 
Given $\al\in\Om^p(M)$ and $\et\in\Om^q(S)$, the identity
\begin{equation}\label{inse}
{i}_{\bar v}\widehat{\al\cdot\et}=\widehat{(({i}_v\al)\cdot\et)} 
\end{equation}
holds for all vector fields $v\in\X(M)$.
Together with \eqref{func}, the last formula permits us to write down an explicit expression for the hat pairing,
\[
\widehat{\al\cdot\et}(\bar v_1,\dots,\bar v_{p+q-k})(f)=\int_S f^*(i_{v_{p+q-k}}\dots i_{v_1}\al)\wedge\et,
\]
which gives a complete description of $\widehat{\al\cdot\et}$ on the open subset of embeddings of $S$ in $M$,
since $\Diff(M)$ acts infinitesimally transitive on embeddings.

From the known fact that differentiation and fiber integration along a boundary free manifold $S$ commute,
we get that the exterior derivative $d$ 
is a derivation for the hat pairing, \ie
\begin{equation}\label{deri}
d(\widehat{\al\cdot\et})= \widehat{(d\al)\cdot\et}+(-1)^{p}\widehat{\al\cdot d\et},
\end{equation}
where $\al\in\Om^p(M)$ and $\et\in\Om^q(S)$.
In the particular case $p+q+1=k$ we get that
$\widehat{(d\al)\cdot\et}+(-1)^{p}\widehat{\al\cdot d\et}=0$.
The identity
\begin{equation*}
L_{\bar v}\widehat{\al\cdot\et}=\widehat{((L_v\al)\cdot\et)}\text{ for all }v\in\X(M) 
\end{equation*}
is a consequence of \eqref{inse} and \eqref{deri}.

\begin{remark}\label{gene}
{\rm 
We need also more general vector fields of type $\ph\bar v$ on $C^\oo(S,M)$,
for $v\in\X(M)$ and $\ph\in C^\oo(S)$. This means that $(\ph\bar v)(f):s\in S\mapsto\ph(s)v(f(s))\in T_{f(s)}M$
for all $f\in C^\oo(S,M)$. 
Contraction and Lie derivative for such vector fields can be written as:
\begin{align}\label{tria}
{i}_{\ph\bar v}\widehat{\al\cdot\et}&=\widehat{({i}_v\al)\cdot\ph\et}\nonumber \\
L_{\ph\bar v}\widehat{\al\cdot\et}&=\widehat{(L_v\al)\cdot\ph\et}-(-1)^{|\al|}\widehat{(i_v\al)\cdot(d\ph\wedge\et)}.
\end{align}
The last identity follows from the following calculation:
\begin{multline*}
L_{\ph\bar v}\widehat{\al\cdot\et}
=di_{\ph\bar v}\widehat{\al\cdot\et}+i_{\ph\bar v}d\widehat{\al\cdot\et}
=d(\widehat{i_v\al\cdot\ph\et})+i_{\ph\bar v}(\widehat{d\al\cdot\et}+(-1)^{|\al|}\widehat{\al\cdot d\et})\\
=\widehat{(di_v\al\cdot\ph\et)}+(-1)^{|\al|-1}\widehat{(i_v\al\cdot d(\ph\et))}
+\widehat{(i_vd\al\cdot\ph\et)}+(-1)^{|\al|}\widehat{(i_v\al\cdot\ph d\et)}\\
=\widehat{(L_v\al\cdot\ph\et)}-(-1)^{|\al|}\widehat{(i_v\al\cdot d\ph\wedge\et)}.
\end{multline*}
}
\end{remark}


\subsection{The twisted cotangent bundle}\label{cotg}

We start with some facts about the canonical symplectic form on the cotangent bundle
of a (possibly infinite dimensional) smooth manifold \cite{MR99}.

Let $M$ be a smooth manifold, $\om$ the canonical symplectic form on the cotangent bundle $T^*M$,
and $p:T^*M\to M$ the natural projection.
This means that $\om=-d\th$, where $\th$ is the Liouville form, so $\th(Y_\al)=\langle\al,Tp.Y_\al\rangle$,
for all $Y_\al\in T_\al(T^* M)$.

Given $v\in\X(M)$, we denote by $X_v\in\X(T^*M)$ the Hamiltonian vector field 
with Hamiltonian function $\ev_v$ on $T^*M$ defined through evaluation at the vector field $v$. 

\begin{lemma}\label{doua}
The vector fields $v$ on $M$ and $X_v$ on $T^*M$ are $p$-related, so $Tp.X_v=v\o p$.
The map $\Ph:v\in\X(M)\mapsto \ev_v\in C^\oo(T^*M)$ is a Lie algebra homomorphism
for the Poisson bracket induced by the canonical symplectic form on $T^*M$. 
\end{lemma}

Given $\al\in\Om^1(M)$, we denote by $\vl(\al)\in\X(T^*M)$ the vertical lift defined by the map $\al:M\to T^*M$:
$$
\vl(\al):\be_x\mapsto \frac{d}{dt}\Big|_{t=0}(\be_x+t\al(x))\text{ for }\be_x\in T^*_xM.
$$ 
We notice that the Liouville form vanishes on vertical lifts: $\th(\vl(\al))=0$,
and the flow of a vertical lift is
\begin{equation}\label{flow}
\Fl_t^{\vl(\al)}(\be)=\be+t\al.
\end{equation}

\begin{lemma}\label{una}
For any 1-form $\al$ on $M$, $i_{\vl(\al)}\om=-p^*\al$, where $\om$ is the canonical symplectic form.
\end{lemma}

\begin{proof}
Let $Y$ be an arbitrary vector field on $T^*M$. 
We show that $\om(\vl(\al),Y)=(p^*\al)(Y)$ at an arbitrary point $\be\in T^*M$:
\begin{align*}
\om(\vl(\al),Y)(\be)&=i_Yd(\th(\vl(\al)))(\be)-i_{\vl(\al)}d(\th(Y))(\be)+\th([\vl(\al),Y])(\be)\\
&=-\frac{d}{dt}\Big|_{t=0}(\th(Y))(\be+t\al)+\th([\vl(\al),Y])(\be)\\
&=-\frac{d}{dt}\Big|_{t=0}\langle\be+t\al,T_{\be+t\al}p.Y\rangle+\langle\be,T_\be p.[\vl(\al),Y]\rangle\\
&=-\langle\al,T_\be p.Y\rangle-\langle\be,\frac{d}{dt}\Big|_{t=0}(T_{\be+t\al}p.Y)\rangle+\langle\be,T_\be p.[\vl(\al),Y]\rangle
\\
&=-\langle\al,T_\be p.Y\rangle.
\end{align*}
The last step uses the identity $\frac{d}{dt}\Big|_{t=0}T_{\be+t\al}p.Y(\be+t\al)=T_\be p.[\vl(\al),Y](\be)$,
which follows from the expression \eqref{flow} of the flow for a vertical lift.
\end{proof}


\begin{remark}\label{patru}
{\rm The Lie algebra homomorphism $\Ph$ from Lemma \ref{doua} can be extended to
an injective Lie algebra homomorphism defined on the semidirect product Lie algebra $\X(M)\ltimes C^\oo(M)$:
\[
\Ph:v+f\in\X(M)\ltimes C^\oo(M)\mapsto \ev_v+p^*f\in C^\oo(T^*M).
\]
}
\end{remark}

\paragraph{Twisted symplectic form}
The canonical symplectic form $\om$ on the cotangent bundle can be twisted by adding a magnetic term, 
namely the pull-back of a closed 2-form $B$ on the base manifold \cite{MR99}. 
One gets a new symplectic form $\om_B=\om+p^*B$ on $T^*M$.
Such twisted symplectic forms also appear in cotangent bundle reduction. 
Let $\{\ ,\ \}_B$ denote the Poisson bracket defined with the twisted symplectic form. 

On the other hand every closed 2-form $B$ defines a Lie algebra 2-cocycle $B$ on $\X(M)$ with values in the module $C^\oo(M)$.
The Lie bracket on the abelian Lie algebra extension $\X(M)\ltimes_B C^\oo(M)$ with characteristic cocycle $B$
is
\[
[v+f,w+g]=[v,w]+L_vg-L_wf+B(v,w).
\]

\begin{proposition}\label{cinci}
The map 
\[
\Ph:\X(M)\ltimes_B C^\oo(M)\to C^\oo(T^*M), \quad \Ph(v)=\ev_v,\quad\Ph(f)=p^*f
\]
is an injective Lie algebra homomorphism for the twsted Poisson bracket $\{\ ,\ \}_B$.
\end{proposition}

\begin{proof}
Given $v\in\X(M)$ and $f\in C^\oo(M)$, the Hamiltonian vector field on $T^*M$ with Hamiltonian function $\ev_v$ is $X_v+\vl(i_vB)$, 
and the Hamiltonian vector field on $T^*M$ with Hamiltonian function $p^*f$ is  $\vl(df)$, because of
the following computations, that use lemma \ref{doua} and lemma \ref{una}:
\begin{gather*}
i_{X_v+\vl(i_vB)}(\om_B)=i_{X_v}\om+i_{X_v}p^*B+i_{\vl(i_vB)}\om=d\ev_v+p^*i_vB-p^*i_vB=d\ev_v\\
i_{\vl(df)}(\om_B)=i_{\vl(df)}\om=p^*df=d(p^*f).
\end{gather*}

The map  $\Phi$ is a Lie algebra morphism because of the identities 
\begin{equation*}
\{\ev_v,\ev_w\}_B=\ev_{[v,w]}+p^*B(v,w),\quad
\{\ev_v,p^*f\}_B=p^*L_vf, \quad \{p^*f,p^*g\}_B=0.
\end{equation*}
The first identity follows from the computation 
\begin{align*}
\{\ev_v,\ev_w\}_B
&=\om_B(X_w+\vl(i_wB),X_v+\vl(i_vB))\\
&=\om(X_w,X_v)+\om(\vl(i_wB),X_v)+\om(X_w,\vl(i_vB))+(p^*B)(X_w,X_v)\\
&=\{\ev_v,\ev_w\}-i_{X_v}p^*i_wB+i_{X_w}p^*i_vB+p^*(B(w,v))\\
&=\ev_{[v,w]}+p^*B(v,w),
\end{align*}
and the other two identities follow in a similar fashion.
\end{proof}



{\footnotesize

\bibliographystyle{new}

\addcontentsline{toc}{section}{References}
\begin{thebibliography}{300}

\bibitem{AlSe10}
A. Alekseev, P. Severa, Equivariant cohomology and current algebras (preprint).

\bibitem{AlSt05}
A. Alekseev and T. Strobl,
{Current algebras and differential geometry}, 
\textit{JHEP} {\bf 03} (2005), 035.

\bibitem{zabzine} 
G. Bonelli and M. Zabzine,
From current algebras for p-branes to topological M-theory,
{\it JHEP} {\bf 09} (2005) 015.

\bibitem{KS96}
Y. Kosmann-Schwarzbach, From Poisson algebras to Gerstenhaber algebras, 
\textit{Ann. Inst. Fourier} {\bf 46} (1996), 1241--1272.

\bibitem{MR99}
J.E. Marsden and T. Ratiu,
{\it Introduction to Mechanics and Symmetry}, 2nd edition, 
Springer, 1999.

\bibitem{Ne04}
K.--H. Neeb,
{Abelian extensions of infinite-dimensional Lie groups},
\textit{Travaux Math.} {\bf XV} (2004), 69--194.

\bibitem{Ne07}
K.-H. Neeb, Lie groups of bundle automorphisms and their extensions, 
in {\it Developments and Trends in Infinite-Dimensional Lie Theory}, 
Eds. K.-H. Neeb and A. Pianzola,
Progress in Mathematics, Vol. 288, Part 2, 2011, pp. 281--338.

\bibitem{nekrasov} 
N. Nekrasov, Lectures on curved beta-gamma systems, pure spinors, and anomalies,
(preprint).  

\bibitem{Vi09}
C. Vizman, Natural differential forms on manifolds of functions, 
\textit{Archivum Math.} {\bf 47}  (2011), 201--215.

\end{thebibliography}

}

\end{document}